\documentclass[12pt,UKenglish]{amsart}
\usepackage[margin=25 mm,bottom=35mm,footskip=15mm,top=35mm]{geometry}
\usepackage{enumerate,url,amssymb,amsthm,amsmath,xstring,enumitem,hyperref,subcaption,floatrow,setspace}
\usepackage[UKenglish]{babel}
\usepackage[numbers]{natbib}
\usepackage{xcolor}
\usepackage{graphicx}

\usepackage{aliascnt}
\usepackage[capitalise,noabbrev]{cleveref}
\crefname{appsec}{Appendix}{Appendices}
\crefformat{equation}{#2(#1)#3}

\theoremstyle{plain}

\newtheorem{theorem}{Theorem}[section]

\newaliascnt{proposition}{theorem}
\newtheorem{proposition}[proposition]{Proposition}
\aliascntresetthe{proposition}

\newaliascnt{lemma}{theorem}
\newtheorem{lemma}[lemma]{Lemma}
\aliascntresetthe{lemma}

\newaliascnt{claim}{theorem}

\aliascntresetthe{claim}

\newaliascnt{corollary}{theorem}

\aliascntresetthe{corollary}

\newaliascnt{conjecture}{theorem}
\newtheorem{conjecture}[conjecture]{Conjecture}
\aliascntresetthe{conjecture}

\newaliascnt{observation}{theorem}

\aliascntresetthe{observation}

\newtheorem*{question*}{Question}

\theoremstyle{definition}
\newaliascnt{definition}{theorem}

\aliascntresetthe{definition}

\newaliascnt{question}{theorem}

\aliascntresetthe{question}

\newaliascnt{example}{theorem}

\aliascntresetthe{example}

\theoremstyle{remark}
\newtheorem*{remark}{Remark}

\crefname{theorem}{Theorem}{Theorems}
\crefname{proposition}{Proposition}{Propositions}
\crefname{lemma}{Lemma}{Lemmas}
\crefname{claim}{Claim}{Claims}
\crefname{corollary}{Corollary}{Corollaries}
\crefname{conjecture}{Conjecture}{Conjectures}
\crefname{observation}{Observation}{Observations}
\crefname{definition}{Definition}{Definitions}
\crefname{question}{Question}{Questions}
\crefname{example}{Example}{Examples}

\Crefname{theorem}{Theorem}{Theorems}
\Crefname{proposition}{Proposition}{Propositions}
\Crefname{lemma}{Lemma}{Lemmas}
\Crefname{claim}{Claim}{Claims}
\Crefname{corollary}{Corollary}{Corollaries}
\Crefname{conjecture}{Conjecture}{Conjectures}
\Crefname{observation}{Observation}{Observations}
\Crefname{definition}{Definition}{Definitions}
\Crefname{question}{Question}{Questions}
\Crefname{example}{Example}{Examples}

\newcommand{\xqed}[1]{%
	\leavevmode\unskip\penalty9999 \hbox{}\nobreak\hfill
	\quad\hbox{\ensuremath{#1}}}
\newcommand{\Endofdef}{\xqed{\lozenge}}

\title[Monochromatic paths in hypercubes]{Monochromatic Paths and a Topological Approach to Norine's Conjecture}

\author[Adam~D\v{z}avoronok]{Adam D\v{z}avoronok}
\address{Department of Applied Mathematics, Charles University, Faculty of Mathematics and Physics, Malostransk\'e n\'am.~25, 118~00 Praha~1, Czech Republic}
\email{adam.dzavoronok@mff.cuni.cz}
\begin{document}
	
\begin{abstract}
Motivated by Norine's conjecture, this paper investigates monochromatic antipodal paths in $2$-edge-coloured hypercubes and simplicial complexes. Our main method relies on a topological criterion applied to triangulated $2$-skeleta. We show that any antipodal colouring of a centrally symmetric, simply connected $2$-complex yields a monochromatic path linking an antipodal pair of vertices. By symmetrically triangulating opposite square faces in certain classes of colourings, we obtain a topological verification of Norine's conjecture for these classes. We also establish quantitative bounds for cases in which only a limited number of square faces present structural obstructions.
\end{abstract}

	\maketitle

\section{Introduction}

Norine's conjecture concerns edge-colourings of the hypercube that are compatible with the antipodal symmetry. Let \(Q_n\) be the graph on \(\{0,1\}^n\), where two vertices are adjacent if they differ in exactly one coordinate, and let
\[
    \tau(u_1,\dots,u_n)=(1-u_1,\dots,1-u_n)
\]
be the antipodal involution. We call \(\tau(v)\) the \emph{antipode} of \(v\). A \(2\)-\emph{edge-colouring} of \(Q_n\) is a function \(E(Q_n)\rightarrow \{\textcolor{red}{red},\textcolor{blue}{blue}\}\). It is \emph{antipodal} if opposite edges receive opposite colours, that is, if the colours of \(uv\) and \(\tau(u)\tau(v)\) are different for every edge \(uv\). A \emph{geodesic} is a shortest path between two vertices.

Norine conjectured that every antipodal \(2\)-edge-colouring of \(Q_n\) contains a monochromatic path from some vertex to its antipode. This conjecture is closely related to several stronger or more flexible formulations. We shall consider the following family of conjectures.

\begin{conjecture}\label{NorineConj}
    Let \(Q_n\) be an \(n\)-dimensional hypercube. Then:
    \begin{enumerate}
        \item (Norine~\cite{Norine}) In any antipodal \(2\)-edge-colouring, there exists a monochromatic path connecting some vertex to its antipode.
        \item In any antipodal \(2\)-edge-colouring, there exists a monochromatic geodesic connecting some vertex to its antipode.
        \item (Feder--Subi~\cite{FederSubi}) In any \(2\)-edge-colouring, there exists a path connecting some vertex to its antipode with at most one colour switch.
        \item (Leader--Long~\cite{LeaderLong}) In any \(2\)-edge-colouring, there exists a geodesic connecting some vertex to its antipode with at most one colour switch.
     \end{enumerate}
\end{conjecture}

All conjectures have been verified in small dimensions. Feder and Subi~\cite{FederSubi} proved the result for \(n\le 5\), West and Wise~\cite{WestWise} handled \(n=6\), Frankston and Scheinerman~\cite{FrankstonScheinerman} proved the case \(n=7\) using SAT solvers, and Kirchweger, Peitl, Subercaseaux, and Szeider~\cite{kirchweger} extended this to \(n=8\). There has also been progress on quantitative variants. Leader and Long~\cite{LeaderLong} proved that every \(2\)-edge-colouring of \(Q_n\) contains a monochromatic geodesic of length \(\lceil n/2\rceil\), which gives an antipodal geodesic with at most \(n/2\) colour switches. Dvo\v{r}\'ak~\cite{Dvorak} improved this to \(\frac{3}{8}n+o(n)\) colour switches, Kirchweger et al.~\cite{kirchweger} further improved the bound to \(0.3125n+O(1)\), and Hollom~\cite{hollom2026hypercubegeodesicscolourchanges} recently obtained \((\frac{\pi}{2}+o(1))\sqrt n\) colour switches.

\subsection*{Preliminaries}
A finite \emph{simplicial complex} $\Delta$ is a collection of finite sets, called simplices, closed under taking subsets. We write $|\Delta|$ for its geometric realisation. We realise all complexes in sufficiently high dimensional Euclidean space to avoid any self-intersections. The graph formed by the vertices and edges of $\Delta$ is called its $1$-skeleton. The $2$-skeleton of $\Delta$ is obtained by also adding all $2$-simplices, that is, all filled triangles. Similarly, a \emph{square complex} is a finite $2$-dimensional cell complex whose $2$-cells are combinatorial squares attached along cycles of length four. We write $|K|$ for the underlying topological space of a square complex $K$.  We denote by $Q_n^{(2)}$ the square complex obtained from $Q_n$ by filling all its $4$-cycles. Throughout, by a \emph{polyhedron} we mean a topological space homeomorphic to the geometric realisation of a finite simplicial complex. This also covers the realisation of square complexes.

A simplicial or square complex is called \emph{centrally symmetric} if it is equipped with a fixed-point-free involution $\tau$ on its vertices which extends to an involution of the whole complex. We always assume that no simplex, and no square face, contains both a vertex $v$ and its antipode $\tau(v)$. In particular, $\tau$ acts freely on the underlying space. An edge-colouring of the $1$-skeleton is called \emph{antipodal} if the edges $uv$ and $\tau(u)\tau(v)$ receive opposite colours.

If $(X,\tau)$ and $(Y,\nu)$ are spaces with fixed-point-free involutions, a continuous map $f:X\to Y$ is called $\mathbb Z_2$-equivariant if \(f\circ \tau=\nu\circ f\). In this paper, the target space will usually be the circle $S^1$ equipped with the antipodal involution $x\mapsto -x$.

\subsection*{Our contribution} The purpose of this paper is to develop a topological approach to Norine's conjecture. The guiding observation is that square faces of the hypercube encode the local obstruction to finding monochromatic antipodal paths. If a square has two opposite vertices that are already connected by a monochromatic path, then one may think of the corresponding diagonal as being available in that colour. Adding such diagonals turns suitable square complexes into centrally symmetric simplicial complexes, to which equivariant topology machinery can be applied. We first prove the following general topological criterion for simplicial complexes.

\begin{theorem}\label{thm:triangulated_squares}
 Let \(\Delta\) be a centrally symmetric simplicial complex such that \(|\Delta|\) is simply connected. Then every antipodal \(2\)-edge-colouring of the edges of the \(1\)-skeleton of \(\Delta\) contains a monochromatic path connecting some antipodal pair of vertices.
\end{theorem}

We briefly outline the proof. Suppose, for a contradiction, that no monochromatic antipodal path exists. The red and blue connected components then allow us to construct a continuous \(\mathbb Z_2\)-equivariant map from \(|\Delta|\) to a unit circle \(S^1\), where \(S^1\) is equipped with the antipodal involution. This contradicts the standard Borsuk--Ulam type fact that no such equivariant map exists from a simply connected centrally symmetric complex to \(S^1\).

To apply this criterion to hypercubes and square complexes, we use square faces. A square is called \emph{resolved} if some pair of opposite vertices of the square is connected by a monochromatic path in the ambient hypercube otherwise, it is called \emph{unresolved}. If all square faces of a centrally symmetric square complex are resolved, then we can triangulate each square by a diagonal witnessed by the corresponding monochromatic path. This gives an antipodal colouring of a centrally symmetric simplicial complex, and Theorem~\ref{thm:triangulated_squares} applies.
In particular, we obtain the following hypercube consequence.

\begin{theorem}\label{thm:no_unresolved_squares}
Let $n\ge3$. If an antipodal $2$-edge-colouring of $Q_n$ has no unresolved square faces, then there exists a monochromatic path joining some vertex to its antipode.
\end{theorem}
This viewpoint is related to a result of Feder and Subi~\cite{FederSubi}, who proved that in any $2$-colouring without alternately coloured $4$-cycles, there is a monochromatic path connecting an antipodal pair. Our Theorem~\ref{thm:no_unresolved_squares} extends this result to a wider class of antipodal colourings.

The next natural step is to make the criterion quantitative. The proof has two main steps. The first is to build on the auxiliary $\mathbb Z_2$ mapping from the proof of Theorem~\ref{thm:triangulated_squares}. We extend it to square complexes with only a few unresolved squares. The second step is to construct a suitable small square complex with appropriate topological properties for the Borsuk-Ulam type argument. Our test complex will be a square subcomplex of $Q_n^{(2)}$ that is formed by gluing two topological discs with antipodal boundaries and $\binom{n}{2}$ squares. By further averaging through translations of this complex, we obtain the quantitative bound.

\begin{theorem}\label{quantitative_feder_subi}
Let $n\ge3$. If an antipodal $2$-edge-colouring of $Q_n$ contains fewer than $3\cdot2^{n-2}$ unresolved squares, then there exists a monochromatic path joining some vertex to its antipode.
\end{theorem}

Hence, any counterexample to Norine's conjecture must contain at least $3\cdot2^{n-2}$ unresolved square faces.

\section{Monochromatic paths in simplicial complexes}\label{sec:topological-criterion}
The goal of this section is to prove \cref{thm:triangulated_squares}. We first recall the following result about simply connected simplicial complexes.
\begin{lemma}\label{lem:no-equivariant-map}[\cite{Borsuk-Ulam}Proposition~5.3.2]
Let \(X\) be a connected finite polyhedron with a fixed-point-free involution
\(\tau\). If \(X\) is simply connected, then there is no continuous
\(\mathbb Z_2\)-equivariant map \(X\to S^1\), where \(S^1\) has the antipodal
involution.
\end{lemma}
\begin{proof}
The statement may be viewed as a generalisation of the familiar consequence of the Borsuk–Ulam theorem that there is no antipodal-equivariant map \(S^2\to S^1\). The essential point is that \(|\Delta|\) is simply connected, so any circle-valued map on \(|\Delta|\) can be unwrapped: instead of recording only an angle modulo \(2\pi\), one can choose a continuous real-valued angle. Thus, if an equivariant map \(f:|\Delta|\to S^1\) existed, we could write
\[
f(x)=e^{i\theta(x)}
\]
for some continuous function \(\theta:|\Delta|\to\mathbb R\). Equivariance would then imply that passing from \(x\) to its antipodal point changes this lifted angle by an odd multiple of \(\pi\). Since this change is continuous and takes values in a discrete set, it must be constant on \(|\Delta|\). But applying the involution twice returns to the original point, while the lifted angle would have changed by twice an odd multiple of \(\pi\), a nonzero real number. This shows why no such equivariant map can exist.
\end{proof}

Next, we use an obstruction formulated in terms of $\mathbb Z_2$-equivariant maps. The next proposition extracts such a map from any colouring with no monochromatic antipodal path.
\begin{proposition}\label{prop:existence-equivariant-map}
Let $\Delta$ be as described in \cref{thm:triangulated_squares}, and let $\Delta_2$ be its $2$-skeleton. If there is no monochromatic path connecting a vertex to its antipode, then there exists a $\mathbb Z_2$-equivariant continuous map $g\colon |\Delta_2| \to S^1$, where $S^1$ is equipped with the antipodal involution.
\end{proposition}
\begin{proof}
First, note that $|\Delta|$ is simply connected if and only if $|\Delta_2|$ is as well. Let us denote the red components (including singleton vertices) by $R_1,R_2,\dots,R_m$ and the blue components canonically by $B_1,B_2,\dots,B_m$, so that $\tau(R_i)=B_i$. Since each vertex belongs to exactly one red and one blue component, we assign to every vertex $v$ a pair $(i(v),j(v))$ such that $v\in R_{i(v)}\cap B_{j(v)}$. Denote this function by $\kappa:V(\Delta_2)\rightarrow [m]^2$.

We now construct a vertex map $F\colon V(\Delta_2) \to \mathbb{R}^2\setminus\{0\}$. Choose $m$ distinct points $p_1,\dots,p_m$ on the upper open semicircle of the unit circle in $\mathbb{R}^2$ for instance, $p_k=(\cos\theta_k,\sin\theta_k)$ with $0<\theta_k<\pi$. Define
\[
F(v)=p_{i(v)}-p_{j(v)}.
\]
Because $i(v)\neq j(v)$ for all vertices, we have $F(v)\neq 0$. Indeed, if $i(v)=j(v)$, then $v$ and $\tau(v)$ lie in the same red component, contradicting the assumption that there is no monochromatic antipodal path. Moreover, $F$ is equivariant since $(i(\tau(v)),j(\tau(v)))=(j(v),i(v))$.

Next, we show that $F$ extends linearly over each simplex without hitting the origin. First consider an edge. Its endpoints lie in a common monochromatic component, and the same half-plane argument as below shows that the image segment avoids the origin.

Now let $\sigma=\{v_1,v_2,v_3\}$ be a $2$-simplex. Among its three edges, at least two have the same colour, and those two edges already involve all three vertices. Hence either all three vertices lie in the same red component or all three lie in the same blue component. The two cases are symmetric, so assume that all three vertices lie in the same red component $R_c$. Then $i(v_1)=i(v_2)=i(v_3)=c$, and the image points are
\[
F(v_k)=p_c-p_{j(v_k)} \qquad (k\in\{1,2,3\}).
\]
Taking the dot product with $p_c$ gives
\[
(p_c-p_{j(v_k)})\cdot p_c = 1-p_{j(v_k)}\cdot p_c>0,
\]
because the points $p_i$ are distinct unit vectors on the semicircle. Thus all three image points lie in the open half-plane $\{x\in\mathbb R^2 : x\cdot p_c>0\}$, so their convex hull avoids the origin. Consequently, $F$ extends affinely over every simplex of $\Delta_2$ to a continuous map
\[
f\colon |\Delta_2|\to\mathbb R^2\setminus\{0\}.
\]
Moreover, this affine extension is equivariant: if \(x=\sum_r\lambda_r v_r\) in a simplex, then \(\tau(x)=\sum_r\lambda_r\tau(v_r)\) and hence \(f(\tau(x))=-f(x)\). Finally, normalising gives a continuous map
\[
g(x)=\frac{f(x)}{\|f(x)\|}\colon |\Delta_2|\to S^1.
\]
It follows that $g(\tau(x))=-g(x)$ for all $x\in|\Delta_2|$. Hence $g$ is $\mathbb Z_2$-equivariant.
\end{proof}

\begin{proof}[Proof of \cref{thm:triangulated_squares}]
Assume, for a contradiction, that no monochromatic path connects a vertex to its antipode. By \cref{prop:existence-equivariant-map}, there exists a continuous $\mathbb Z_2$-equivariant map $g\colon |\Delta_2|\to S^1$. This contradicts \cref{lem:no-equivariant-map}, because $\Delta_2$ is centrally symmetric and simply connected whenever $\Delta$ is. Hence a monochromatic path connecting an antipodal pair of vertices must exist.
\end{proof}

\section{Monochromatic paths in hypercubes}

In this section we first translate \cref{thm:triangulated_squares} back to the original hypercube problem. This is the point at which the global definition of an unresolved square is used: every resolved square can be triangulated by a diagonal whose colour is witnessed by an existing monochromatic path in the cube. After that we proceed with the proof of Theorem~\ref{quantitative_feder_subi}.

\begin{proposition}\label{prop:unresolved_square_subcomplex}
Let \(n\ge3\), and let \(K\) be a centrally symmetric square subcomplex of the \(2\)-skeleton of \(Q_n\). Suppose that \(|K|\) is simply connected. If an antipodal \(2\)-edge-colouring of \(Q_n\) contains no unresolved square belonging to \(K\), then there exists a monochromatic path connecting an antipodal pair in \(Q_n\).
\end{proposition}

\begin{proof}
Since no square of \(K\) is unresolved, every square \(S\subseteq K\) has a pair of opposite vertices \(a,b\in S\) connected by a monochromatic path in \(Q_n\). For each antipodal pair of squares \(S,\tau(S)\), choose such a pair \(a,b\) in \(S\), add the diagonal \(ab\), and colour it by the colour of a chosen monochromatic path from \(a\) to \(b\). In the opposite square \(\tau(S)\), add the diagonal \(\tau(a)\tau(b)\) and colour it with the opposite colour. Since $n\ge3$, no square is $\tau$-fixed.

This gives an antipodal colouring of the \(1\)-skeleton of a centrally symmetric triangulation \(\Delta\) of \(K\). The underlying space of \(\Delta\) is \(|K|\), hence is simply connected. Therefore \cref{thm:triangulated_squares} gives a monochromatic path in \(\Delta\) joining some vertex \(v\) to \(\tau(v)\). Replacing each added diagonal on this path by the corresponding monochromatic path in \(Q_n\), we obtain a monochromatic walk in \(Q_n\) from \(v\) to \(\tau(v)\). Removing repetitions gives a monochromatic path.
\end{proof}

\begin{proof}[Proof of \cref{thm:no_unresolved_squares}]
Apply \cref{prop:unresolved_square_subcomplex} with \(K=Q_n^{(2)}\). For completeness we sketch that \(|Q_n^{(2)}|\) is simply connected for \(n\ge2\). Every loop in the \(1\)-skeleton of \(Q_n\) corresponds to a word in coordinate flips in which each coordinate occurs an even number of times. The square faces of \(Q_n^{(2)}\) give the commutation relations between flips in distinct coordinates. By commuting equal flips next to each other and cancelling adjacent pairs, every loop contracts. Thus \(|Q_n^{(2)}|\) is simply connected, and the proposition applies.
\end{proof}

As mentioned in the Introduction, the proof of Theorem~\ref{quantitative_feder_subi} has two main steps. We first construct a small topological disc as a square subcomplex of \(Q_n^{(2)}\). We now identify the vertices of \(Q_n\) with subsets of \([n]\), and we write the antipodal map as complementation:
\(
\tau(A)=[n]\setminus A .
\)

For \(1\le i<j\le n\), put
\[
A_{i,j}=\{1,\dots,i-1\}\cup\{j+1,\dots,n\}.
\]
Let \(S_{i,j}\) be the square face of \(Q_n\) with vertices
\[
A_{i,j},\qquad
A_{i,j}\cup\{i\},\qquad
A_{i,j}\cup\{j\},\qquad
A_{i,j}\cup\{i,j\}.
\]
Let \(D_n\) be the square subcomplex consisting of all squares \(S_{i,j}\) with \(1\le i<j\le n\). 

\begin{lemma}\label{sorting_disk}
The complex \(D_n\) is a topological disc. Its boundary is the cycle
\[
\varnothing,\{1\},\{1,2\},\dots,[n],
\{2,\dots,n\},\{3,\dots,n\},\dots,\{n\},\varnothing .
\]
In particular, \(D_n\) is simply connected and has \(\binom n2\) square faces.
\end{lemma}

\begin{proof}
We prove by induction that $D_n$ is a disc. For \(n=2\), the complex \(D_2\) consists of one square, so the claim is clear. Suppose \(n\ge 3\). The squares \(S_{i,j}\) with \(j<n\) all have \(n\) in every vertex. After deleting the coordinate \(n\), these squares form a copy of \(D_{n-1}\). By induction, this part is a disc.

The remaining squares are \(S_{i,n}\), where \(1\le i<n\). They form a strip of \((n-1)\) squares between the two paths
\[
\varnothing,\{1\},\{1,2\},\dots,\{1,\dots,n-1\}
\]
and
\[
\{n\},\{1,n\},\{1,2,n\},\dots,[n].
\]
This strip is a disc. It meets the embedded copy of \(D_{n-1}\) exactly along the boundary path
\[
\{n\},\{1,n\},\{1,2,n\},\dots,[n].
\]
Therefore \(D_n\) is obtained by gluing two discs along a boundary path, and is again a disc. Tracing the remaining boundary gives precisely the displayed cycle.
\end{proof}

\begin{remark}
The disc \(D_n\) can be interpreted as the commutation diagram associated with sorting the word \(1,2,\dots,n\) into \(n,n-1,\dots,2,1\). Indeed, the square \(S_{i,j}\) records the elementary commutation between the two crossings involving coordinates \(i\) and \(j\): performing these two independent operations in one order or the other gives the two length-two paths around the square. In this sense, the \(\binom n2\) square faces of \(D_n\) record the \(\binom n2\) inversions that have to be created in order to reverse the word. The boundary cycle is formed by the two monotone extreme sorting paths. Thus, \(D_n\) is a planar disc filling the commutation diagram between these two paths.
\end{remark}

Furthermore, let \(E_n=D_n\cup\tau(D_n)\).

\begin{lemma}\label{lem:En}
For $n \ge3$, the complex \(E_n\) is centrally symmetric and simply connected.
\end{lemma}
\begin{proof}
The two discs \(D_n\) and \(\tau(D_n)\) meet exactly in their common boundary. Indeed, every vertex of
$D_n$ has the form $\{1,\ldots,a\}\cup\{b,\ldots,n\}$, and its complement is
the interval $\{a+1,\ldots,b-1\}$. Such an interval is again of this form only
in the boundary cases, namely when it is an initial segment, a final segment,
$\emptyset$, or $[n]$. Hence
\[
D_n\cap \tau(D_n)=\partial D_n .
\]
In particular, the intersection is connected. Since both $D_n$ and $\tau(D_n)$
are simply connected by \cref{sorting_disk}, the Seifert-van Kampen theorem
implies that $E_n$ is simply connected. It is centrally symmetric by
construction. Moreover, this implies that $E_n$ has at most $2\binom n2=n(n-1)$ square faces. 
\end{proof}
We now adapt the construction of the maps $f$ and $g$ from Proposition~\ref{prop:existence-equivariant-map}. There, an important step was to show that $f$ avoids the origin. In the presence of unresolved squares, this is no longer automatic. We will use the freedom in choosing the cyclic order of the points $p_1,p_2,\dots,p_m$ on the semicircle. We record the following elementary lemma.
\begin{lemma}\label{lemma:4 point ordering}
Let $a_1,a_2,b_1,b_2$ be four distinct points on the unit circle. Then
\[
0\in \operatorname{conv}\{a_i-b_j:i,j\in\{1,2\}\}
\]
if and only if the pairs $\{a_1,a_2\}$ and $\{b_1,b_2\}$ interlace on the circle equivalently, up to relabelling and reversal of the cyclic order, the order is $a_1,b_1,a_2,b_2$.
\end{lemma}
\begin{proof}
We first show that the convex-hull condition is equivalent to the intersection of the two chords $a_1a_2$ and $b_1b_2$, which is equivalent to interlacing.

Suppose that $0\in \operatorname{conv}\{a_i-b_j:i,j\in\{1,2\}\}$. Then there are coefficients $\lambda_{ij}\ge 0$ with $\sum_{i,j}\lambda_{ij}=1$ such that
\[
0=\sum_{i,j}\lambda_{ij}(a_i-b_j).
\]
Put $\alpha_i=\lambda_{i1}+\lambda_{i2}$ and $\beta_j=\lambda_{1j}+\lambda_{2j}$. Then $\alpha_1+\alpha_2=\beta_1+\beta_2=1$, and the previous identity gives $\sum_i\alpha_i a_i=\sum_j\beta_j b_j$. Hence the point $x:=\sum_i\alpha_i a_i=\sum_j\beta_j b_j$ lies in both chords $[a_1,a_2]$ and $[b_1,b_2]$. Thus the two chords intersect.

Conversely, suppose that the chords intersect, and choose $x\in [a_1,a_2]\cap [b_1,b_2]$. Then we may write $x=\alpha_1a_1+\alpha_2a_2=\beta_1b_1+\beta_2b_2$, where all coefficients are nonnegative and $\alpha_1+\alpha_2=\beta_1+\beta_2=1$. Define $\lambda_{ij}=\alpha_i\beta_j$. Then $\lambda_{ij}\ge 0$, $\sum_{i,j}\lambda_{ij}=1$, and
\[
\sum_{i,j}\lambda_{ij}(a_i-b_j)=\sum_i\alpha_i a_i-\sum_j\beta_j b_j=x-x=0.
\]
Therefore, $0$ belongs to the convex hull of the four points $a_i-b_j$.
\end{proof}
We now prove a strengthening of Proposition~\ref{prop:existence-equivariant-map}.
\begin{proposition}\label{prop:2-unresolved}
Let $K$ be a simply connected and centrally symmetric square subcomplex of $Q_n^{(2)}$. Suppose that at most two antipodal pairs of squares are unresolved. Then there exists a monochromatic path connecting an antipodal pair in \(Q_n\).
\end{proposition}
\begin{proof}
For each resolved square and its antipodal mate, add the corresponding diagonals in an antipodal way, as in Proposition~\ref{prop:unresolved_square_subcomplex}. Suppose, for a contradiction, that there is no monochromatic path connecting an antipodal pair. Define the component map $\kappa$ as in Proposition~\ref{prop:existence-equivariant-map} for the red and blue components \(R_1,\dots,R_m\) and \(B_1,\dots,B_m\) of the original colouring of \(Q_n\), not just the components of $K$. We choose distinct points $p_1,\dots,p_m$ on the open upper semicircle, together with an ordering $\pi$ of these points, and obtain the corresponding map $F_\pi$.

Consider an unresolved square with vertices labelled $v_{00},v_{10},v_{01},v_{11}$. Its boundary colours must alternate otherwise one of the two length-two paths between opposite vertices would be monochromatic. After possibly exchanging the colours and relabelling the square, we may therefore assume that the horizontal edges are red and the vertical edges are blue. Thus, after relabelling the components if necessary, we have
\[
\kappa(v_{00})=(a,b),\qquad \kappa(v_{11})=(c,d),\qquad
\kappa(v_{10})=(a,d),\qquad \kappa(v_{01})=(c,b).
\]
Unresolvedness gives $a\neq c$ and $b\neq d$. The absence of a monochromatic antipodal path gives $i(v)\neq j(v)$ for every vertex $v$ applied to the four displayed vertices, this gives $a\neq b$, $a\neq d$, $c\neq b$, and $c\neq d$. Hence the four indices $a,b,c,d$ are pairwise distinct.

Fix the positions of the points on the semicircle and choose their order uniformly at random. For one unresolved square, the two pairs $\{p_a,p_c\}$ and $\{p_b,p_d\}$ interlace with probability $1/3$. Since there are at most two antipodal pairs of unresolved squares, the union bound gives an ordering for which no corresponding pair of pairs interlaces. Fix such an ordering $\pi$.

We now extend $F_\pi$ continuously to a map $f:|K|\to \mathbb R^2\setminus\{0\}$. On each resolved square, we extend affinely over the two triangles obtained by adding the chosen diagonal. As in Proposition~\ref{prop:existence-equivariant-map}, this extension avoids the origin. On an unresolved square with vertices $v_{00},v_{10},v_{01},v_{11}$, Lemma~\ref{lemma:4 point ordering} and the choice of $\pi$ imply that the convex hull of
\[
F(v_{00})=p_a-p_b,\quad
F(v_{10})=p_a-p_d,\quad
F(v_{01})=p_c-p_b,\quad
F(v_{11})=p_c-p_d
\]
avoids the origin. We may therefore extend over this square, for instance by bilinear interpolation. Defining the extension on antipodal squares antipodally gives an equivariant map.

After normalising, we obtain a $\mathbb Z_2$-equivariant map $g:|K|\to S^1$. This contradicts Lemma~\ref{lem:no-equivariant-map}. Therefore a monochromatic path connecting an antipodal pair must exist.
\end{proof}
We finish the proof of Theorem~\ref{quantitative_feder_subi} with the following averaging lemma.
\begin{lemma}\label{lem:avg}
Let $n\ge3$, let \(U\) be a centrally symmetric set of square faces of \(Q_n\), and let \(G=\mathbb F_2^n\rtimes S_n\) be the group of cube automorphisms generated by translations and coordinate permutations. If \(|U|<3\cdot2^{n-2}\), then there exists $g\in G$ such that $|gE_n^{(2)}\cap U|\le 4$.
\end{lemma}
\begin{proof}

Every element of \(G\) commutes with the antipodal map \(\tau\), and \(G\) acts transitively on the square faces of \(Q_n\).
Choose \(g\in G\) uniformly at random. Since \(E_n\) is centrally symmetric and simply connected, so is \(gE_n\). For a fixed square face \(S\) of \(Q_n\), we have
\[
\Pr(S\in gE_n^{(2)})=
\frac{|E_n^{(2)}|}{\#\{\text{square faces of }Q_n\}},
\]
where \(E_n^{(2)}\) denotes the set of square faces of \(E_n\). The number of square faces of \(Q_n\) is
\(
\binom n2 2^{n-2},
\)
since one chooses the two varying coordinates and then fixes the remaining \(n-2\) coordinates. Since \(|E_n^{(2)}|\le 2\binom n2\), it follows that
\[
\Pr(S\in gE_n^{(2)})\le
\frac{2\binom n2}{\binom n2 2^{n-2}}
=
2^{3-n}.
\]

By the linearity of expectation,
\[
\mathbb E |U\cap gE_n^{(2)}|
=
\sum_{S\in U}\Pr(S\in gE_n^{(2)})
\le
|U|\cdot2^{3-n}
<6.
\]
The set \(U\) is invariant under the antipodal map by assumption, and \(gE_n^{(2)}\) is antipodally invariant, since $g$ commutes with $\tau$. Since \(n\ge 3\), no square face of \(Q_n\) is fixed by the antipodal map. Hence, \(U\cap gE_n^{(2)}\) is a disjoint union of antipodal pairs of squares, and therefore
\[
|U\cap gE_n^{(2)}|\in\{0,2,4,6,\dots\}.
\]
Since the expectation is less than \(6\), there exists \(g\in G\) such that
\(
|U\cap gE_n^{(2)}|\leq 4.
\)
\end{proof}

\begin{proof}[Proof of \cref{quantitative_feder_subi}]
Let \(U\) be the set of unresolved square faces of \(Q_n\). This set is centrally symmetric. By Lemma~\ref{lem:avg}, there exists \(g\in G\) such that
\(
|U\cap gE_n^{(2)}|\le 4.
\)

Since this intersection is centrally symmetric and no square face is fixed by the antipodal map for \(n\ge3\), the complex \(gE_n\) contains at most two antipodal pairs of unresolved squares. By Lemma~\ref{lem:En}, \(gE_n\) is centrally symmetric and simply connected. Proposition~\ref{prop:2-unresolved} therefore gives a monochromatic path connecting an antipodal pair in \(Q_n\).
\end{proof}

\section{Concluding remarks}\label{sec:remarks}

We have shown that unresolved square faces form the local obstruction to the
topological criterion. In particular, any counterexample to Norine's conjecture
must contain at least \(3\cdot 2^{n-2}\) unresolved square faces.

The proof of \cref{quantitative_feder_subi} uses only the size of the set of
unresolved squares, through averaging over copies of the test complex \(E_n\).
This is unlikely to be the whole story. The square faces of a hypercube satisfy
many higher-dimensional relations: for example, in every \(3\)-dimensional
subcube, the six square faces form the boundary of a cube. Thus one expects
that unresolved squares in a genuine counterexample cannot be placed
arbitrarily, but must obey additional closure or parity-type constraints coming
from these local cube boundaries.

It would be interesting to make this structural obstruction precise. One
possible direction is to use several test complexes at once, or to modify the
test complex inside higher-dimensional subcubes, in order to detect not only
the number of unresolved squares but also their arrangement.

\subsection*{Acknowledgements}
I would like to thank V\'it Jel\'inek, Jan Kyn\v{c}l, and Pavel Valtr for introducing me to this problem during the Seminar on combinatorial problems at Charles University. I thank the other participants, in particular Matou\v{s} \v{S}afr\'anek, for early discussions. I also thank Mykhaylo Tyomkyn for his helpful comments on the manuscript.
 	\bibliographystyle{plainnat}
	\bibliography{Hypercube-ref}

\end{document}